\documentclass{llncs}
\usepackage[german,english]{babel}
\usepackage{times,theorem}
\usepackage{latexsym,color,graphics}
\usepackage{diagrams1}
\diagramstyle[tight,centredisplay%
,dpi=600
,UglyObsolete
,heads=LaTeX%
]

\usepackage{epsfig}
\usepackage{amssymb}
\usepackage{amsmath}
\usepackage{amsfonts}
\usepackage{xspace}
\usepackage{graphicx}
\begingroup
\catcode`\~=11
\gdef\urltilde{\lower 0.6ex\hbox{~}}
\endgroup

 \renewcommand{\L}{\mathcal{L}}

\title{Internal Symmetry Group in Categorial Topology}
\author{Zoran Majki\'c}
\authorrunning{Zoran Majki\'c}
\institute{ISRST, Tallahassee, FL, USA\\
\email{majk.1234@yahoo.com}}

\newtheorem{propo}{Proposition}
\newtheorem{coro}{Corollary}

\begin{document}

\maketitle

\begin{abstract}
 The interdefinability of the universal concepts of category theory has been introduced by Lawvere. The perfect interdefinability between the objects and arrows of some category, defines the class of  Perfectly Symmetric Categories (PSC) where each category can be  represented equivalently by its arrows or by its objects only.
 Such symmetry, differently from the global categorial symmetry ( categorial-symmetry group $CS(\mathbb{Z})$ of all comma-propagation transformations), ia a local internal symmetry inside a given PSC category.

Given a PSC category (as a "geometric object") $\textbf{C}$ we can consider its properties (the categorial commutative diagrams) preserved under  actions  of a particular endofunctor $E$ which transforms any commutative diagram into an invariant "up to isomorphism" diagram. We show that this kind of internal categorial invariance is a phenomena of a local categorial symmetry under an Internal Catergorial Symmetry group $ICS(\mathbb{N})$ of all local enfdofunctorial transformations. Then we establish the  relationships between this local internal symmetry and global general symmetry between n-dimensional levels (the comma categories obtained from a PSC category $\textbf{C}$) . We show that if a base category $\textbf{C}$ is a PSC, then all its ne-dimensional levels are PSC as well.
\end{abstract}

\section{Introduction to Symmetry and Categorial Topology of n-dimensional Levels}
In Categorial Topology, given a category (as a "geometric object") \cite{Cart01,GrDi60,Vacu21,DeRi15}
we can consider its properties preserved under continuous action (a "deformation") of a some transformations.

However, the Metacategory space, valid for all  categories,  can not be defined by using well-know Grothendeick's approach with discrete ringed spaces as demonstrated in \cite{Majk24G}. However we can consider any category $\textbf{C}$ as an abstract geometric object,
 that is, a discrete space where the points  are the objects of this category and arrows between objects as the paths. Based on this approach, we  can define the Cat-vector space $V$ valid for all categories with addition operation for the vectors, and their inner product. For the categories where we define the norm ("length") of the vectors in $V$ we can define also the outer (wedge) product of the vectors in $V$ and we show that such Cat-algebra satisfies two fundamental properties of the Clifford geometric algebra \cite{Majk24G}.

Most related philosophical questions deal with specific symmetries, objectivity,
interpreting limits on physical theories, classification, and laws of nature \cite{Shaw18,Eins36}. The recent history of the philosophy of mathematics is largely focused on grasping and defining the nature and essence of mathematics and its objects \cite{Korz94}.
Categorial symmetries has been introduced by my second PhD thesis, 27 years ago, with a number of applications in Computer Science, with 6wo kinds of symmetries:
\begin{itemize}
  \item \emph{Internal symmetry} of a given category between its objects and arrows;
  \item \emph{global symmetry} valid \emph{for all} categories, that is for the whole category theory composed by its principal concepts (concrete categories, functors and natural transformations).  This symmetry is based on the transformation of three principal categorial concepts called "comma-propagation".
\end{itemize}
For a long time I did not publish them. However,  I have published  a book on Data Integration \cite{Majk14} by using exclusively category theory formalism and introduced these ideas of categorial \emph{internal symmetry} in it. Finally, after all this work finally I extended my PhD thesis in recent book \cite{Majk23} both with revisited theoretical part and significant applications in Computer Science (for natural deduction, lambda-calculus, process-algebras and database theory).
Moreover, after that I have published  the papers \cite{Majk23s,Majk25a} dedicated to formalization of the symmetry group of the global symmetry.

 So, in this paper I will extend the work about \emph{internal symmetry}  by formal definitions (non provided in  \cite{Majk23}) of its category-symmetry group $ICS$.
 It is well known that, in a given category $\textbf{C}$, its arrows in $Mor_\textbf{C}$ can be composed by the \emph{partial} operation $\circ:Mor_\textbf{C}\times Mor_\textbf{C} \rightarrow Mor_\textbf{C}$  to obtain more complex arrows, and for each object $a \in Ob_\textbf{C}$,  the identity arrow $id_a \in Mor_\textbf{C}$ is the "representation" of this object. From this point of view, the properties of the category can be represented by using \emph{only} its arrows.
Such compositional property for the objects and the possibility to represent the category \emph{only} by using the objects of a category generally does not hold.

That is, generally we have no the symmetry between the arrows and the objects of a given category. Our aim is also to investigate the class of categories which are \emph{conceptually symmetric}, that is, where each arrow can be equivalently represented by some object of this category, and in which we can define the partial operation $*:Ob_\textbf{C}\times Ob_\textbf{C} \rightarrow Ob_\textbf{C}$, able to represent the "composition of the objects" of the category. This symmetry appear as an \emph{arrows-to-objects transformation} and is a kind of  \emph{internal symmetry}, called  "conceptual symmetry" as well.

This kind of \emph{internal} categorial symmetry is denominated also as a conceptual categorial symmetry of transformation of the arrows into the objects of the category preserving the properties of commutative diagrams (composition of the arrows).  A category $\textbf{C}$ for which such internal transformation is possible is called \emph{perfectly symmetric} category (PSC), for which there exist the duality operator from the morphisms (arrows) into the objects of this category
\begin{equation}\label{dualop}
B_\top:Mor_\textbf{C} \rightarrow Ob_\textbf{C}
\end{equation}
which, applied to an arrow $f$  returns with its "conceptualized object" $\widetilde{f} = B_\top(f)$.

The most general internal category, with a minimum requirements, called \emph{perfectly symmetric categories}, is determined by only duality operator $B_T$ in (\ref{dualop}) which must also satisfy the following properties:
\begin{definition}\label{def:PerfectSym}
A   $\textbf{C}$ is  \emph{perfectly symmetric category} (PSC) if there exists
a duality operator $B_\top$ from arrows  into the objects (which is dual type of categorial primitive concepts) and hence if there exist the composition $g\circ f $ of two arrows in  $\textbf{C}$, then we have the following homomorphism $B_\top:(Mor_\textbf{c}, \circ) \rightarrow (\widetilde{Ob}_\textbf{C}, *)$, where $\widetilde{Ob}_\textbf{C} \subset Ob_\textbf{C}$ is the subset of conceptualized objects in $\textbf{C}$ such that,
\begin{equation}\label{eq:SymHomomorph}
B_\top(g\circ f) = B_\top(g)* B_\top(f)
\end{equation}
where $B_\top(g)$ and $B_\top(f)$ are the "conceptualized objects" in $\textbf{C}$ denote shortly by $\widetilde{g}$ and $\widetilde{f}$ respectively. The duality operatot $B_\top$ must satisfy the  "representability principle" that for each identity arrow $id_a:a\rightarrow a$ it must be valid the following isomorphism in $\textbf{C}$ (for each object $a$ the conceptualized object obtained from its identity arrow must be isomorphic to it):
\begin{equation} \label{eq:representatib}
is_a:B_\top(id_a) \simeq a
\end{equation}
so that $id_a= is_a\circ is_a^{-1}:a\rightarrow a$.
\end{definition}
The objects obtained from arrows will be called
\emph{conceptualized} objects. For example let us consider the
polynomial function $f(x) =(x^2-3x +1)$, with the domain being the
closed interval of reals from 0 to 1, and codomain $\mathbb{R}$ of
all reals. It is an arrow $f:[0,1] \rightarrow \mathbb{R}$ in the
category $\textbf{Set}$. The conceptualized object obtained from
this arrow, denoted by $\widetilde{f}$, is the \emph{set} (hence an
object in $\textbf{Set}$) equal to the graph of this polynomial
function. That is, $\widetilde{f} = \{(x, f(x))~ |~ x \in [0,1] \} \in Ob_{Set}$. So, the $\textbf{Set}$) is a perfectly symmetric category where the partial operation of composition $*$ used in (\ref{eq:SymHomomorph}) is just the composition of binary relations.

 In general, all constructions of algebraic topology are functorial; the notions of category, functor and natural transformation originated here. The arrow categories are more simple forms of the \emph{comma} categories and were introduced by Lawvere \cite{Lawve63} in the context of the interdefinability of the universal concepts of category theory. The basic idea is the elevation of arrows of one category $\textbf{C}$ to objects in another.

It is well known that, given a base category $\textbf{C}$, we can represent its morphisms as objects by using a derived \emph{arrow} category $\textbf{C}\downarrow \textbf{C}$ (a special case of the comma category), such that for a given arrow $f:a \rightarrow b$ between two objects $a,b \in Ob_C$, we have the object denoted by $\langle a,b,f\rangle  \in Ob_{C\downarrow C}$ in this arrow category.
 With this, we introduce the infinite hierarchy of the arrow categories, for a given base category  $\textbf{C}$, as follows.
The notion of hierarchy of common arrow categories for a given base category $\textbf{C}$ is inductively defined for $n\geq 1$ by
      \begin{equation}\label{eq:n-levels0}
      (\textbf{C}\downarrow\textbf{C})^1 =_{def} \textbf{C}\downarrow\textbf{C},~~~~~~  (\textbf{C}\downarrow\textbf{C})^{n+1} =_{def} (\textbf{C}\downarrow\textbf{C})^n \downarrow (\textbf{C}\downarrow\textbf{C})^n
      \end{equation}
and, derived from it, the \emph{n-dimensional levels} based \emph{only} on the base category $\textbf{C}$, instead on the its first arrow category $(\textbf{C}\downarrow\textbf{C})$ used in (\ref{eq:n-levels0}) for $n\geq 1$, denoted by
      $\textbf{C}_{n+1} =_{def} (\textbf{C}\downarrow\textbf{C})^{n}$
That is, the n-dimensional levels are defined inductively for $n\geq 1$ by
\begin{equation}\label{eq:n-levels0s2}
\textbf{C}_1 =_{def}  \textbf{C},  ~~~~~~~~\textbf{C}_{n+1}=_{def} \textbf{C}_{n}\downarrow\textbf{C}_{n}
\end{equation}
In what follows we will use the following standard comma projection functors for all n-dimensional levels of a given base category $\textbf{C}$, and $n\geq 1$:
\begin{itemize}
  \item The first comma projection, $F_{st} = (F_{st}^0,F_{st}^1):\textbf{C}_{n+1}\rightarrow \textbf{C}_{n}$,
      such that:\\
       for each object $\langle a,b,f\rangle$ in $\textbf{C}_{n+1}$,  $F_{st}(\langle a,b,f\rangle) = F_{st}^0(\langle a,b,f\rangle) = a$;\\
       for each arrow $(h_1,h_2)$ in $\textbf{C}_{n+1}$,  $F_{st}(h_1,h_2) = F_{st}^1(h_1,h_2) = h_1$.
  \item The second  projection, $S_{nd} = (S_{nd}^0,S_{nd}^1):\textbf{C}_{n+1}\rightarrow \textbf{C}_{n}$,
      such that:\\
       for each object $\langle a,b,f\rangle$ in $\textbf{C}_{n+1}$,  $S_{nd}(\langle a,b,f\rangle) = S_{nd}^0(\langle a,b,f\rangle) = b$;\\
       for each arrow $(h_1,h_2)$ in $\textbf{C}_{n+1}$,  $S_{nd}(h_1,h_2) = S_{nd}^1(h_1,h_2) = h_2$.
   \item The natural transformation  $\psi:F_{st} \rTo^\centerdot S_{nd}$.\\
              Note that a natural transformation $\psi$ associate to every object $X$ an arrow $\psi_X:F_{st}^0(X) \rightarrow S_{nd}^0(X)$, and this mapping can be represented by the function denoted by
       \begin{equation} \label{projectionTrnasf}
       J^{-1}:Ob_{\textbf{C}_{n+1}} \rightarrow Mor_{\textbf{C}_{n}}
       \end{equation}
For example, we have that $J^{-1}(\langle a,b,f\rangle) = \psi_{\langle a,b,f\rangle} =f$.
 \item We introduce the \emph{encapsulation} operator operator $J:Mor_{\textbf{C}_{n}}\rightarrow Ob_{\textbf{C}_{n+1}}$, as operator inverse to the operator $J^{-1}$, such that for each arrow $f:a\rightarrow b$ in $\textbf{C}_{n}$, we obtain       $J(f) = \langle a,b,f\rangle$.
 \item The \emph{diagonal} functor ,
\begin{equation} \label{eq:diagFunc}
\blacktriangle = (\blacktriangle^0, \blacktriangle^1):\textbf{C}\rightarrow(\textbf{C}\downarrow\textbf{C})
\end{equation}
such that for any object $a \in \textbf{C}$, it holds that $\blacktriangle^0(a) = J(id_a) = \langle a,a, id_a\rangle$, and for each arrow $f:a\rightarrow b$, $\blacktriangle^1(f) = (f;f)$.
\end{itemize}
The next proposition \cite{Majk23} is important for the relationship between n-dimensional levels and symmetrical categories (for the conceptually closed categories defined in what follows:
\begin{propo} \label{prop:A2}
For each n-dimensional level $\textbf{C}_{n}$, $n\geq 2$, with its identity endofunctor $Id_{\textbf{C}_{n}}$,
there exist the natural transformations $\sigma:\blacktriangle \circ F_{st}\rTo^\centerdot Id_{\textbf{C}_{n}}$ and $\sigma^{-1}:Id_{\textbf{C}_{n}}\rTo^\centerdot \blacktriangle\circ S_{nd}$, such that\footnote{We recall that for the simplicity we consider the natural transformations as a function from objects into arrows that are components of this natural transformation.
}
\begin{equation} \label{eq:sigmaNatTransf}
 \sigma = (\psi \blacktriangle^0F_{st}^0;\psi), ~~~\emph{and} ~~~\sigma^{-1} = (\psi; \psi \blacktriangle^0S_{nd}^0)
 \end{equation}
  so  wit vertical composition $\sigma^{-1}\bullet\sigma = (\psi;\psi)$.
 \end{propo}
 For example, for any object $J(f)$ in $\textbf{C}_{2}$, of an arrow $f$  in $\textbf{C}_{1} = \textbf{C}$, we obtain that
 \begin{equation} \label{eq:sigmaNatTransf11}
 \sigma(J(f)) = (id_{dom(f)};f), ~~~\emph{and} ~~~\sigma^{-1}(J(f)) = (f;id_{cod(f)})
 \end{equation}
 where $dom(f)$ and $cod(f)$ are the domain and codomain objects of the arrow $f$. So, \\$(\sigma^{-1}\bullet\sigma)(J(f)) = \sigma^{-1}(J(f))\circ \sigma(J(f)) = (f;id_{cod(f)})\circ (id_{dom(f)};f)\\ = (f\circ id_{dom(f)}:id_{cod(f)}\circ f) = (f;f)$.

The reason that we introduced n-dimensional level topology is to make in evidence \emph{geometric interpretations} \cite{Majk23,Majk23s,Majk25a} of the hierarchy of comma categories obtained from any given category $\textbf{C}$. In this case in the $\textbf{C}_1$  each arrow is just representable as a linear vector, that is 1-dimensional level. Any arrow $(h_1;k_1):J(f_1)\rightarrow J(g_1)$ in $\textbf{C}_2$  is geometrically representable by bidimensional commutative rectangle   in  $\textbf{C}$), so $\textbf{C}_2$ represents 2-dimensional level. An arrow $((l_1;l_2);(l_3;l_4)):J(h_1;k_1)\rightarrow J(h_2;k_2)$ in $\textbf{C}_3$  is geometrically representable by 3-dimensional commutative box, so $\textbf{C}_3$ represents 3-dimensional level, etc...

By introduction of the n-dimensional levels, which are the arrow categories, we need to consider the relationship between the isomorphic arrows in the base category $\textbf{C}$ and its arrow category $\textbf{C}_2 = \textbf{C}\downarrow\textbf{C}$. Consequently, from the fact that the objects of $\textbf{C}\downarrow\textbf{C}$ are obtained from the arrows of $\textbf{C}$ the isomorphic arrows in $\textbf{C}\downarrow\textbf{C}$ introduce the new concept of a kind of "isomorphism of arrows", together with the standard concept of isomorphism of objects, and this new kind of "isomorphism" we can introduce by the following definition \cite{Majk23}:
\begin{definition} \label{def:newIsomorph}
Let $h:a\rightarrow c$ and $k:b \rightarrow d$ be two arrows of $\textbf{C}$. We tell that these two arrows $h$ and $k$ are "\emph{equal up to isomorphism}", denoted by $$h ~\cong ~k$$ iff there is the (standard) isomorphism\footnote{It is easy to verify that each isomorphic arrow in $\textbf{C}\downarrow\textbf{C}$ is a pair of isomorphic arrows of $\textbf{C}$. Thus, in this case we have two isomorphic arrows in $\textbf{C}$: $is_1:a\rightarrow b$ and $is_2:c\rightarrow d$.}
$is =(is_1;is_2):J(h) \rightarrow J(k)$ (satisfying the \emph{equality} $k\circ is_1 = is_2\circ h$) of the objects  $J(h)$ and $J(k)$ of the arrow category $\textbf{C}\downarrow\textbf{C}$.
\end{definition}
 What is important for functorial semantics of derivation of the objects in a given category $\textbf{C}$ from its arrows, is that in this case we can use  the arrow categories $\textbf{C}\downarrow\textbf{C}$ where the arrows $g:a\rightarrow b$ of $\textbf{C}$ are encapsulated as the objects $J(g)$ in $\textbf{C}\downarrow\textbf{C}$
 so the functorial representation has to be given by the functors $T_e = (T_e^0, T_e^1):\textbf{C}\downarrow\textbf{C}\rightarrow \textbf{C}$.

 Consequently, based on this hierarchy of the n-dimensional levels,  for perfectly symmetric category $\textbf{C}$ with given duality operator $B_T = T_e^0 J$, we can represent the functorial categorial symmetry in a given n-dimensional level, for $n\geq 1$, by the following diagram:
\begin{equation} \label{eq:funct-symmetry}
\begin{diagram}
 a         &\rTo^{f}          & b    & && J(f)= \langle a,b,f\rangle && &&\widetilde{f} = B_T(f)\\
    \dTo^{h_1}       & &               \dTo_{h_2} &\Leftrightarrow &&  \dTo^{(h_1,h_2)} && T_e~ \mapsto &&  \dTo_{T_e(h_1;h_2)}\\
 c  &  \rTo^{g}               &   d & &&J(g)= \langle c,d,g\rangle && && \widetilde{g}= B_T(g)\\
  & in ~\textbf{C}_n  && && in ~\textbf{C}_{n+1}=\textbf{C}_n\downarrow\textbf{C}_n &&&&in ~\textbf{C}_n
\end{diagram} ~~~~~~
\end{equation}
As we can see, in order to satisfy such a categorial symmetry  property, not only the "conceptualized" objects $\widetilde{f}$ and $\widetilde{g}$ (obtained from the respective arrows $f:a\rightarrow b$ and $g:c\rightarrow d$) must exist in $\textbf{C}_n$, but they also must coherently represent the rule of these arrows\footnote{ Especially, if $f = f_2\circ f_2$ is a composed arrow in $\textbf{C}$,  how the 'conceptualized' object $\widetilde{f_2\circ f_2}$ is structurally composed by its 'conceptualized' components $\widetilde{f}_1$ and $\widetilde{f}_2$.
},
and  the arrow $T_e(h_1;h_2)$ in  $\textbf{C}_n$ must exist as well.


Obviously, this functor $T_e$ does not exist generally for each perfectly symmetric category, and our work is to define the subclass of perfectly symmetric categories $\textbf{C}$ where such a construction is possible. Clearly, the categories that support such functorial "reflection" of arrows into its objects are more expressive and powerful.
So, it has been provided \cite{Majk23} the following hierarchy of internally symmetric categories from the bottom to the top symmetry level (where the lower class of internally symmetric categories contain the higher levels of internally symmetric categories:
\begin{enumerate}
  \item Perfectly  symmetric  (with duality operator $B_\top$) categories (PSC)
  \item  Conceptually closed (by a functor $T_e$) categories (CoCC)
  \item Symmetry-extended categories (SEC)
  \item Imploded categories (IMC)
\end{enumerate}
with IMC $\subset$ SEC $\subset$ CoCC $\subset$ PSC,
where the imploded categories are considered as the upper bound symmetry case of the n-dimensional levels, and represent the collapse of all n-dimensional levels into their basic category. If we consider, for example, the (n+1)-dimensional level as the metatheory of the n-dimensional level, then in the imploded category we can express all its metatheories.

  Consequently, in Categorial Topology, given a category (as a "geometric object") $\textbf{C}$ we can consider its properties preserved under continuous action (a "deformation") of a symmetry group of transformations. Indeed, we can consider any category $\textbf{C}$ as an abstract geometric object \cite{Majk24G}, that is a discrete space where the points of such abstract space are the objects of this category and arrows between objects as the paths: given any two points (two objects of the category) we can have a number of oriented paths from first to the second point, some of them equal (commutative diagrams in this category between these two objects).

 In geometry, groups arise naturally in the study of symmetries and geometric transformations: The symmetries of an object form a group, called the symmetry group of the object, and the transformations of a given type form a general group. Lie groups appear in symmetry groups in geometry, and also in the Standard Model of particle physics. The Poincare group is a Lie group consisting of the symmetries of spacetime in special relativity. Point groups describe symmetry in molecular chemistry.

In order to have a self-contained paper for all people that have no my book \cite{Majk23}, in the Section 2 will be provided a short mathematical definition of internal symmetry hierarchy introduced above, and used the same notations  and symbols used in the book. The original contributions, with formal definition of the internal symmetry group of transformations valid for whole hierarchy, are then presented in  the Section 3.
\section{Introduction of the Hierarchy of Internal Symmetries}
The concept of a topology, conceptually closed categories,  has been introduced in my PhD thesis ~\cite{Majk98} with a number of its applications.

The minimal requirements for internal symmetry of a category are provided by the existence of the duality operator $B_T$ with the properties specified by Definition \ref{def:PerfectSym}, which defines the class of perfectly symmetric categories (PSC) and, as an example of PSC, we provided in previous section the $\textbf{Set}$ category. If we introduce another requirements, then we generate a more specific subclasses of PSC.
So,  we can begin to define the hierarchy of internal symmetries of the categories, as follows.\\\\
\textbf{1. Conceptually Closed Categories (CoCC)}:
\\
This set of categories is a strict subset of perfectly symmetric categories heaving the duality operator $B_T$.
This kind of more specific category symmetry is obtainable when there exists descending functor $T_e=(T_e^0,T_e^1):\textbf{C}\downarrow\textbf{C} \rightarrow \textbf{C}$ (because of that we call this symmetry  "\emph{closed} by functor"), as it was anticipated by figure (\ref{eq:funct-symmetry}), inverse (in $\textbf{Cat}$ ) to the diagonal functor  $\blacktriangle =(\blacktriangle^0,\blacktriangle^1):\textbf{C}\longrightarrow (\textbf{C}\downarrow \textbf{C})$ introduced in (\ref{eq:diagFunc}).

In this case
the duality operator $B_\top$ introduced for the conceptually symmetric categories in Definition \ref{def:PerfectSym}, can be defined by
\begin{equation} \label{eq:CoCC}
 B_\top =_{def} T_e^0 J
\end{equation}
 An important subset of symmetric categories are the Conceptually Closed Categories (CoCC), specified as follows: \index{conceptually closed categories (CoCC)}
\begin{propo}  \label{def:symCat} \cite{Majk98}\textsl{A conceptually closed} category (CoCC) is a \emph{perfectly} symmetric category $~\textbf{C}$ for which there exists a functor $T_e = (T_e^0,T_e^1): (\textbf{C}\downarrow \textbf{C})\longrightarrow
\textbf{C}$, "inverse" to diagonal functor $\blacktriangle$, such that for the object-component of this functor   $T_e^0 = B_T \psi$, where $\psi:F_{st} \rTo^\centerdot S_{nd}$ is the natural transformation.\\
Then we obtain the following natural isomorphism:
 \begin{equation} \label{eq:ClosedIso}
 ~\varphi:T_e \circ \blacktriangle~\backsimeq~I_{\textbf{C}} ~~~~~~~~~~\emph{and} ~~~~~~~~~~ \varphi^{-1}:I_{\textbf{C}}~\backsimeq~T_e \circ \blacktriangle
\end{equation}
where $I_\textbf{C}$ is  the identity endofunctor for $\textbf{C}$. Thus, for each arrow n $\textbf{C}$, $f:a\rightarrow b$ between objets $a$ and $b$i, we have also the derived by the functor $T_e$ the arrow
\begin{equation} \label{eq:ClosedIso8}
\varphi(b) \circ T_e^1(f;f)\circ \varphi^{-1}(a):a \rightarrow b
\end{equation}
 between these objects.
\end{propo}
In fact it is easy to show that natural transformation $\varphi$ is isomorphic, that is, for any object $a$ in $\textbf{C}$, $T_e \circ \blacktriangle (a) = T_eJ(id_a) =$ (from (\ref{eq:CoCC}) ) $= B_T(id_a)$, so its component arrow is $\varphi(a):B_T(id_a) \rightarrow a$ just the isomorphism (\ref{eq:representatib}) valid in a perfectly symmetric category.

It has been shown \cite{Majk98,Majk23} that each CoCC introduces two particular natural transformations:
\begin{coro} \label{coro:A15} \cite{Majk23}
Let $\textbf{C}$ be a Conceptual Closed Category. Then there are two particular natural transformations $\tau:F_{st} \rTo^\centerdot T_e$ and $\tau^{-1}:T_e \rTo^\centerdot S_{nd}$, for the first and second projections $F_{st},S_{nd}:\textbf{C}\downarrow\textbf{C}\rightarrow\textbf{C}$ and this new "middle projection" $T_e$ of the conceptual closure, defined by the following vertical composition
\begin{equation} \label{eq:CoCCnat}
\tau = (T_e\sigma)\bullet (\varphi^{-1}F_{st}) ~~~\emph{and}~~~
\tau^{-1} = (\varphi S_{nd})\bullet (T_e\sigma^{-1})
\end{equation}
where the natural transformations $\sigma:\blacktriangle \circ F_{st}\rTo^\centerdot Id_{\textbf{C}}$ and $\sigma^{-1}:Id_{\textbf{C}}\rTo^\centerdot \blacktriangle\circ S_{nd}$ (have the vertical composition $\sigma^{-1}\bullet\sigma =(\psi;\psi)$), and  $~\varphi:T_e \circ \blacktriangle~\backsimeq~I_{C}$ is the natural isomorphism in (\ref{eq:ClosedIso}) with its inverse $\varphi^{-1}$,
and hence
\begin{equation} \label{eq:CoCCnat2}
\tau^{-1}\bullet\tau = (\varphi S_{nd})\bullet (T_e(\psi;\psi))\bullet (\varphi^{-1}F_{st})
\end{equation}
shown by the following diagram
\begin{equation}\label{eq:tauVcomp}
\begin{diagram}
\textbf{C}\downarrow\textbf{C}& &\pile{\rTo^{F_{st}}\\ ~~~~~~\downarrow \tau \\ \rTo~{T_e} \\~~~~~~\downarrow \tau^{-1} \\ \rTo_{S_{nd}}}&& \textbf{C}
\end{diagram}
\end{equation}
\end{coro}
Consequently, if $\textbf{C}$ is a CoCC,  Thus, for each arrow $f:a\rightarrow b$ in $\textbf{C}$, we have also these two arrows  in $\textbf{C}$ for its conceptualized object $\widetilde{f}$:
\begin{multline} \label{eq:tauS}
\tau(J(f))= T_e^1(id_a;f)\circ \varphi^{-1}(a):a\rightarrow \widetilde{f} \\ \tau^{-1}(J(f))=\varphi(b)\circ T_e^1(f;id_a):\widetilde{f}\rightarrow b~~~~~~~~~~~~~~~~~~~~~~~~~~~~~~~~~~~~~~~~~~~~~~~~~~~~~~~~~~~
\end{multline}
It is shown \cite{Majk98,Majk23} that if $\textbf{C}$ is CoCC, then all its higher dimensional levels $\textbf{C}_2, n \geq 2$, are CoCC.

Moreover, if there is also the natural transformation $\varrho:T_e\rTo^\centerdot F_{st}$, then there is also the  adjunction $(\blacktriangle, T_e,\varepsilon_e,\eta_e)$ where the unit is natural isomorphism

$\eta_e = \varphi^{-1}:I_\textbf{C}\rTo^\centerdot T_e \circ \blacktriangle$ and counit $\varepsilon_e: \blacktriangle\circ T_e \rTo^\centerdot I_{\textbf{C}_2}$\\ is defined as a pair of natural transformations
$\varepsilon_e =_{def} (\varrho, ~\psi\bullet\varrho)$. It is shown \cite{Majk23} that in this specific case we can interpret the conceptualized objects in $\textbf{C}$ as a particular kind of \emph{limits}.
\begin{example}: $\textbf{Set}$ category, where each arrow $f:a\rightarrow b$ is a functions, is an example of CoCC in which a conceptualized object is binary relation $\widetilde{f} =B_T(f) = T_e^0(J(f)) =_{def} \{(x,f(x) | x\in a\}$, that is, the graph of the function $f$. \\
In fact,
  an arrow $(k_1;k_2):J(f)\rightarrow J(g)$ in $\textbf{Set}\downarrow\textbf{Set}$ can be represented by the following commutative diagram in $\textbf{Set}$:
\begin{equation}\label{eq:set1}
 \begin{diagram}
\widetilde{f}=T_e^0(J(f))&&J(f)&& a      &     &\rTo^{f}   &    & b   \\
&&&& & & &  &    \\
\dTo^{k=}_{T_e^1(k_1;k_2)}&&\dTo_{(k_1;k_2)}&& \dTo_{k_1} &&  in ~\textbf{Set}  && \dTo_{k_2} \\
&&&& &           &      & &\\
\widetilde{g}=T_e^0(J(g))&&J(g)&&c  &      &  \rTo^{g} &      & d   \\
\textbf{Set}&\lTo_{T_e}&\textbf{Set}\downarrow\textbf{Set}&&       & & &  &
\end{diagram}
\end{equation}
where the function $k =T_e(k_1;k_2):\widetilde{f}\rightarrow \widetilde{g}$, we have that of the any element (a pair), $(x,f(x)) \in \widetilde{f}$, this function $k$ is defined by
\begin{equation} \label{eq:TeSet}
k(x,f(x)) =_{def}~ (k_1(x),k_2(f(x))) = (k_1(x), (k_2\circ f)(x)) ~~~\in \widetilde{g}
\end{equation}
It is easy to show that the functor $T_e:\textbf{Set}\downarrow\textbf{Set}\rightarrow \textbf{Set}$ is well defined:

1. For each identity arrow $(id_a;id_b):J(f)\rightarrow J(f)$ in $\textbf{Set}\downarrow\textbf{Set}$, the arrow $k= T_e(id_a;id_b)$ from definition (\ref{eq:TeSet}), we obtain that $k$ is the identity function $id_{\widetilde{f}}$ of the conceptualized object $\widetilde{f} = T_e(J(f))$.

2. For the composition of the arrows $(k_1;k_2):J(f)\rightarrow J(g)$ and $(h_1;h_2):J(g)\rightarrow J(m)$, by using definition in (\ref{eq:TeSet}), it holds the functorial property $T_e(h_1\circ k_1; h_2\circ k_2) = T_e((h_1;h_2)\circ (k_1;k_2)) = T_e(h_1;h_2)\circ T_e(k_1;k_2)$.\\
\textbf{Remark:} in the case of commutative diagram in (\ref{eq:set1}) when $g=f:a\rightarrow b$, $k_1 =id_a$ and $k_2 = id_b$, we obtain from (\ref{eq:TeSet}), for each pair $(x.x) \in \widetilde{id_a}$,

$k(x,x))=k(x,id_a(x)) = (f(x),f(id_a(x))) = (f(x), f(x))$

$=(f,f)(x,x)$
\\\\
Thus, $T_e^1(f:f) = (f,f)$, where $(f,f):\widetilde{id_a}\rightarrow \widetilde{id_b}$ is the pair of functions $f$.
Consequently, the derived by the functor $T_e$ the arrow (\ref{eq:ClosedIso8}) in $\textbf{Set}$ \emph{is equal} to $f$, i.e.,
\begin{equation} \label{eq:ClosedIso9}
\varphi(b) \circ T_e^1(f;f)\circ \varphi^{-1}(a) = f
\end{equation}
$\square$\\
 In \textbf{Set} the adjunction $(\blacktriangle, T_e,\varepsilon_e,\eta_e)$ is valid, because there exists the natural transformation $\varrho:T_e\rTo^\centerdot F_{st}$  for which each its arrow component  is equal to the first projection function $\pi_1:A\times B\rightarrow A$.  Thus,  for any two arrows $f:a\rightarrow b$ and $g:c\rightarrow d$, with the commutativity $g\circ k_1 =k_2\circ f$ corresponding to the arrow $(k_1; k_2):J(f) \rightarrow J(g)$ in $\textbf{Set}\downarrow\textbf{Set}$,  we obtain the following commutative diagram in $\textbf{Set}$ (the extended version of the left hand part of the diagram in (\ref{eq:set1})),
\begin{diagram}
J(f)&&&&(x,f(x))\in \widetilde{f}  &&&\rTo^{\varrho(J(f))=\pi_1}&&&  a \ni x\\
\dTo_{(k_1;k_2)} &&T_e \mapsto&&\dTo^{k=}_{T_e^1(k_1;k_2)}&&&in~ \textbf{Set}&&&\dTo_{k_1}    \\
J(g)&&&&(k_1(x),k_2f(x))\in \widetilde{g} &&&\rTo^{\varrho(J(g))=\pi_1}&&&  c\ni k_1(x)
\end{diagram}
Based on the definition in (\ref{eq:TeSet}), and $k_2f(x) =gk_1(x)$.
\\$\square$
\end{example}
\textbf{2. Symmetry-extended Categories (SEC)}:\\
This set of categories is a strict subset of CoCC, thus it is closed by the functor $T_e$ as well. So, this symmetry-closure functor $T_e:\textbf{C}\downarrow\textbf{C}\rightarrow\textbf{C}$ generates the conceptualized objects of $\textbf{C}$ from its arrows, with two particular natural transformations represented by the diagram (\ref{eq:tauVcomp}), with their vertical composition in (\ref{eq:CoCCnat2}).

Thus, for each object $J(f)$ of $\textbf{C}\downarrow\textbf{C}$, obtained from the arrow $f:a\rightarrow b$ of $\textbf{C}$, we obtain that

$(\tau^{-1}\bullet\tau) (J(f))= ((\varphi S_{nd})\bullet (T_e(\psi;\psi))\bullet (\varphi^{-1}F_{st}))(J(f))$

$ = (\varphi S_{nd})(J(f))\circ (T_e(\psi;\psi))(J(f))\circ (\varphi^{-1}F_{st})(J(f))$

$ = \varphi(b)\circ T_e^1(f;f) \circ \varphi^{-1}(a)$\\
where $\varphi(a):\widetilde{id_b}\rightarrow b$ and $\varphi^{-1}(a):a\rightarrow \widetilde{id_a}$ are two isomorphic arrows in $\textbf{C}$, (heaving their inverse arrows as well), so that the equation above can be represented by a commutative diagram in $\textbf{C}$ corresponding to an isomorphic arrow in $\textbf{C}\downarrow\textbf{C}$ as follows:
\begin{diagram}
\widetilde{id_a}= T_e^0(J(id_a))&&\rTo^{T_e^1(f;f)}&& \widetilde{id_b}= T_e^0(J(id_b)) && & J(T_e^1(f;f))   \\
 \dTo^{\varphi(a)}&&&&\dTo^{\varphi(b)} &\Leftrightarrow~& & \dTo^{is}_{=(\varphi(a);\varphi(b))} \\
a &&\rTo^{(\tau^{-1}\bullet\tau) (J(f))} && b&&& J((\tau^{-1}\bullet\tau) (J(f)))      \\
& in ~\textbf{C} & &&&&&  in~\textbf{C}\downarrow\textbf{C}~
\end{diagram}
where $is:J(T_e^1(f;f))\rightarrow J((\tau^{-1}\bullet\tau) (J(f)))$ is an isomorphic arrow in  $\textbf{C}\downarrow\textbf{C}$, so that, from Definition \ref{def:newIsomorph}, we have the following two arrows "equal up to isomorphism" in $\textbf{C}$ (from Definition \ref{def:newIsomorph}) :
\begin{equation}\label{eq:eqUpIso}
T_e^1(f;f)~\cong ~(\tau^{-1}\bullet\tau)(J(f))
\end{equation}
So, the specific interesting case of this general result is when $(\tau^{-1}\bullet\tau)(J(f)) = f$, so that $T_e^1(f;f)$ is equal up to isomorphism to the arrow $f$. \\
So we define the symmetry-extended category as a subset of the CoCC that satisfy this specific property for each arrow of this category:
\begin{definition}  \label{def:symCatExtend} \cite{Majk98}
A \textsl{symmetry-extended} category $\textbf{C}$  is CoCC, such that each its isomorphism $T_e^1(f;f)~\cong ~f$  is just an equality $T_e^1(f;f)~= ~f$ of these arrows , that is, when holds that
\begin{equation} \label{eq:ExCoCC}
   ~\tau^{-1}\bullet \tau = \psi:F_{st} \rTo^\centerdot S_{nd}
\end{equation}
for  vertical  composition of natural transformations $\tau: F_{st} \rTo^\centerdot
T_e ~$ and $ \tau^{-1}:T_e \rTo^\centerdot S_{nd}$ defined by (\ref{eq:CoCCnat}).
\end{definition}
So, in such symmetry-extended categories we have that for each arrow $f:a\rightarrow b$ the closure functor $T_e$ satisfies
\begin{equation} \label{eq:ClosedIsoMor}
 ~T_e^1(f;f)= \varphi^{-1}(cod(f))\circ f \circ \varphi(dom(f))
\end{equation}
So, in a symmetry-extended category $\textbf{C}$, for each arrow $(h;k):J(f)\rightarrow J(g)$ (an equation $k\circ f = g\circ h$) in $\textbf{C}\downarrow\textbf{C}$, we obtain the following commutative diagram in $\textbf{C}$:
\begin{equation}\label{fig:ExtenDiag}
\begin{diagram}
J(f)&&  a      &           &\rTo^{f}        &        & b   \\
 &&& \rdTo_{\tau_{J(f)}}& & \ruTo_{\tau^{-1}_{J(f)}} &   \\
 \dTo{(h;k)}&&\dTo^{h}   &    &   \widetilde{f} = T_e^0(J(f)) &  &\dTo_{k}\\
 &&&           &\dTo_{ T_e^1(h;k)}       & &\\
J(g)&&c      &           \rTo^{g}    &         &                    & d   \\
       &&& \rdTo_{\tau_{J(g)}}& & \ruTo_{\tau^{-1}_{J(g)}} &    \\
    &&&        &   \widetilde{g} = T_e^0(J(g)) &  &  \\
\end{diagram}
\end{equation}
which represents the \emph{strong correlation} between an arrow and its conceptualized object: each arrow is equal to composition of two arrow with their "middle" object equal to the conceptulized object of this original arrow, that is,
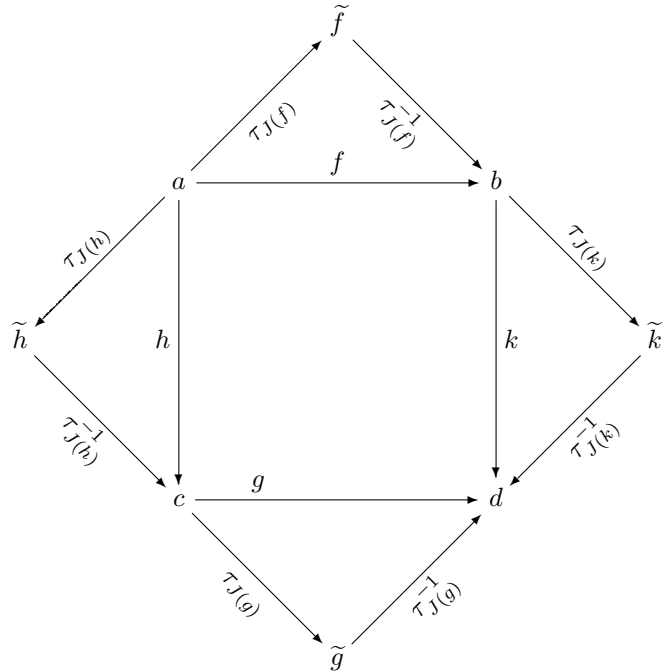
\begin{diagram}
 &&   &    &   \widetilde{f}  &  &&&\\
 & && \ruTo_{\tau_{J(f)}}& & \rdTo_{\tau^{-1}_{J(f)}} && &  \\
 && a      &           &\rTo^{f}        &        & b  && \\
&\ldTo^{\tau_{J(h)}} && & &  && \rdTo^{\tau_{J(k)}}&  \\
 \widetilde{h}&&\dTo^{h}   &    &     &  &\dTo_{k}&&\widetilde{k}\\
&\rdTo_{\tau^{-1}_{J(h)}} &&           &       & & &\ldTo_{\tau^{-1}_{J(k)}}&\\
&&c      &           \rTo^{g}    &         &                    & d   \\
&&       & \rdTo_{\tau_{J(g)}}& & \ruTo_{\tau^{-1}_{J(g)}} &    \\
&&   &        &   \widetilde{g}  &  &  \\
\end{diagram}
and this decomposition can inductively continue for the arrow components  of two natural transformations $\tau$ and $\tau^{-1}$ as well.

In the diagram above, if the objects $a,b,c$ and $d$ belong to an n-dimensional level then the objects $J(f), J(g),J(h), J(k)$ and $J(k\circ f)$ belong to next higher (n+1)-dimensional level, and hence $\widetilde{f},\widetilde{g},\widetilde{h},\widetilde{k}$ and $\widetilde{k\circ f} = \widetilde{k}*\widetilde{f}$ are derived objects hrom this higher level. Consequently, the components of the natural transformation $\tau$ we nominate as "ascending" arrows while the components of the natural transformation $\tau^{-1}$ we nominate as "descending" arrows of the n-dimensional level. So, each symmetry-extended category we can separate in a number of subcategories composed by the objects of the same level of conceptualization, beginning from the first level of basic (non-conceptualized) objects of such a category.

Let us now consider the conceptualized objects obtained from a composed arrow $g\circ f$ with conceptualized composed object

$\widetilde{g\circ f} = T_e^0(J(g\circ f))$

$= B_\top\psi(J(g\circ f))$

$= B_\top(g\circ f)$ ~~~~~~~~~~~~~~~~from identity $\psi J$

$= B_\top(g)* B_\top(f)$ ~~~~~~~~~~~~~~~~from (\ref{eq:SymHomomorph})

$ = T_e^0(J(g))* T_e^0(J(f))$

$= \widetilde{g} * \widetilde{f}$\\
 and for that consider the diagram (\ref{fig:ExtenDiag}) when $k=g$ and $h = f$, so from it we obtain the following commutative diagram
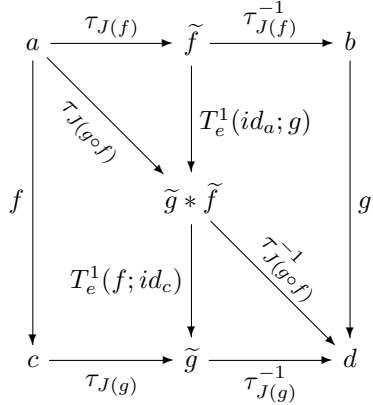
\begin{diagram}
  a      &   \rTo^{\tau_{J(f)}}   &\widetilde{f} &\rTo^{\tau^{-1}_{J(f)}}   & b \\
& ~~~~\rdTo_{\tau_{J(g\circ f)}}& \dTo_{T_e^1(id_a;g)} &  &  \\
 \dTo^{f}   &    &   \widetilde{g}*\widetilde{f}  &  &\dTo_{g}\\
&           &  \dTo^{T_e^1(f;id_c)}     & \rdTo^{\tau^{-1}_{J(g\circ f)}}& \\
 c      &   \rTo_{\tau_{J(g)}}   &\widetilde{g} &\rTo_{\tau^{-1}_{J(g)}}   & d \\
\end{diagram}
where external square diagram is trivial commutative diagram $g\circ f = g\circ f$.

An example of symmetry-extended categories is the category $\textbf{Set}$ as well. In fact, for each arrow $f:a \rightarrow b$ in $T_e^1(f;f)~\cong ~f$, and an element $x \in a$, we have that

$~((\tau^{-1}\bullet \tau)(J(f)))(x) = (\tau^{-1}(J(f))) \tau(J(f))(x)$

$ = (\tau^{-1}(J(f))) (T_e^1(id_a;f)\circ \varphi^{-1}(a))(x) $ from (\ref{eq:tauS})

$ = (\tau^{-1}(J(f))) (T_e^1(id_a;f))(x,x)$

$= \tau^{-1}(J(f))(x, f(x))$

$= (\varphi(b)\circ T_e^1(f;id_a))(x, f(x))$ from (\ref{eq:tauS})

$= \varphi(b)(f(x),f(x)) = f(x)$

$ = (\psi(J(f)))(x)$, \\
so that we obtain $\tau^{-1}\bullet \tau) =\psi$, that is, $(\tau^{-1}\bullet \tau)(J(f)) = f $.\\
Another examples \cite{Majk23} of symmetry-extended category are the category of relations $\textbf{Rel}$, Natural Deduction Category  $\textbf{ND}$ for constructive interpretation of logic implication, the category $\textbf{IC}$ for the propositional intuitionistic calculus and the category of finite labelled trees and the CCC with fixed-point operators.
 A significant example of such  categorial symmetry, applied to database mapping, is the \textbf{DB} category,  developed in \cite{Majk14}.

 It was shown \cite{Majk23} that for the n-dimensional transformations this symmetry-extended property is an invariant property valid globally in all n-dimensional levels.

 This section, dedicated to the internal categorial symmetry of its primitive categorial concepts, arrows and objects, and to the classification (hierarchy) of different kinds of these symmetries, we will conclude with the extreme limit case of the symmetry-extended categories analyzed in previous section. However, this limit case is important only from theoretical point of view, because we can not impose more constraints than that used for this limit case, in order to obtain another subclass of symmetric categories.

 From  the more concrete point of view, this limit case of symmetric categories is practically insignificant, as we will show, because such categories have the very poor structural properties.

 Just because of that, before a specification of this bottom level in the category symmetry hierarchy, as an important  conclusion about internal categorial symmetry and its classification, it is useful to consider the symmetry-extended categories as the most expressive level of category symmetry, with  most constraints that preserves its symmetry property, but  still with great capability to represent a lot of fundamental mathematical structures.

In fact, we have shown  that also the $\textbf{Set}$ category is a symmetry-extended category, and this category is fundamental for the denotational semantics in a significant number of mathematical theories.
 \\
 \\
 \textbf{3. Imploded Categories (IMC)}:\\
 Now we can define the limit case (with the most constraints) of the internal categorial symmetry and, successively, to provide the final hierarchy of the classes of categories with internal symmetry, and their relationship with n-dimensional levels as well.

As we will show, the extreme case of internal categorial symmetry  corresponds to the equivalence between all n-dimensional levels and basic category $\textbf{C}$. In this extreme case, basic category is denominated as "imploded category" from the fact that all n-dimensional levels are flattened (and contained) inside this basic category.

Formal specification of this extreme case of internal categorial symmetry is provided by the following definition: \index{imploded categories}
\begin{definition}\label{def:imploded} \cite{Majk98}
An imploded category $\textbf{C}$ is a CoCC (Definition \ref{def:symCatExtend}) for which there exists the adjoint equivalence $(T_e, \blacktriangle, \varphi,\eta)$ of the closure functor $T_e:\textbf{C}\downarrow\textbf{C}\rightarrow\textbf{C}$, 
and its right adjoint diagonal functor $\blacktriangle:\textbf{C}\rightarrow \textbf{C}\downarrow\textbf{C}$.
\end{definition}
So, in this definition of adjunction, the counit is a natural isomorphism $\varphi$ defined in (\ref{eq:ClosedIso}), while the unit is new natural isomorphism $\eta:Id_{\textbf{C}\downarrow\textbf{C}} \rightarrow \blacktriangle \circ T_e$, with the following adjunction diagram for each $J(f)$ of a given arrow $f:a\rightarrow b$ in $\textbf{C}$  and object $\textbf{d}$ of $C$,
\begin{equation} \label{fig:adjImploded}
\begin{diagram}
 J(f)~  &\rTo^{\eta_{J(f)}} &~ J(id_{\widetilde{f}})=\blacktriangle(\widetilde{f})    &&& &\widetilde{f} = T_e(J(f))~~~~~~~~~~~~~ &\rTo^{T_e(h)~~~~~}&~~~~~\widetilde{id_d}=T_e\blacktriangle(d) \\
 & \rdTo_{h=(h_1;h_2)}&     \dTo_{(\underline{f};\underline{f})=\blacktriangle(\underline{f})} &&&&\dTo^{\underline{f}}&\ldTo_{\varphi_d} &     \\
 &    &  J(id_d)= \blacktriangle(d)   &&&&d\\
            &  \textbf{C}\downarrow\textbf{C} &&&\lTo^{T_e}&&& \textbf{C}
\end{diagram}
\end{equation}
Thus, for any object $J(f)$ of $\textbf{C}\downarrow\textbf{C}$, the pair $(\widetilde{f},\eta_{J(f)})$ is an universal arrow, so that for any other pair $(d,h)$ there exist a unique arrow $\underline{f}$ from $\widetilde{f}$ to $d$ in $\textbf{C}$.

Note that the isomorphic arrow (an arrow component of the natural isomorphism $\eta$), $(is_1;is_2) = \eta_{J(f)}:J(f) \rightarrow J(id_{\widetilde{f}})$ of $\textbf{C}\downarrow\textbf{C}$ is composed by two isomorphic arrows
\begin{equation}\label{eq:adjImploded1}
is_1:a \simeq \widetilde{f}  ~~~~\emph{and}~~~~ is_2:b \simeq \widetilde{f}
\end{equation}
that compose the following commutative diagram:
\begin{equation} \label{fig:adjImploded2}
\begin{diagram}
a= dom(f)&&\rTo^{f}&& b=cod(f) && & J(f)   \\
 \dTo^{is_1}&&&&\dTo^{is_2} &\Leftrightarrow~& & \dTo^{\eta_{J(f)}}_{=(is_1;is_2)} \\
\widetilde{f} &&\rTo^{id_{\widetilde{f}}} && \widetilde{f}&&& J(id_{\widetilde{f}})=\blacktriangle^0 T_e^0(J(f))      \\
& in ~\textbf{C} & &&&&&  in~\textbf{C}\downarrow\textbf{C}~
\end{diagram}
\end{equation}
Hence, from the commutative diagram with two isomorphic arrows $is_1$ and $is_2$, we obtain that
\begin{equation}\label{eq:adjImploded2}
f = is_2^{-1}\circ is_1
\end{equation}
that is, the arrow $f$ must be an isomorphism, $f:a\simeq b$, as well.  That is, the commutative diagram in (\ref{fig:adjImploded2}) represents the isomorphisms
\begin{equation}\label{eq:adjImploded3}
dom(f) \simeq \widetilde{f} \simeq cod(f)
\end{equation}
for every arrow of $\textbf{C}$. From the fact that it holds for every object $J(f)$ of $\textbf{C}\downarrow\textbf{C}$, we have that the arrows of $\textbf{C}$, different from the identity arrows, \emph{must be isomorphic arrows}.

It was shown \cite{Majk23} that each Imploded Category is a Symmetry-extended Category and all its n-dimensional levels, for $n\geq 2$,  are imploded categories as well and hence this conceptual symmetry is globally valid.\\
\textbf{Remark}: In order to understand immediately how the structural properties of the limit case of Imploded Categories are poor, consider that, from the fact that they are also symmetry-extended categories, all arrows of the commutative diagram (\ref{fig:ExtenDiagB}) are the isomorphisms !  Consequently, the "minimal"  imploded categories are the most simple (structurally) \emph{discrete categories}. In fact, the skeleton of every imploded category is just a discrete category, where there is exactly the \emph{bijection} between the arrows and the objects of a category.
\\$\square$\\
 We recall that all n-dimensional levels $\textbf{C}_n$, for $n\geq 2$, of a given imploded category $\textbf{C}$ are equivalent to this base category. This fact explains why the imploded categories represent the last level in the hierarchy of internal symmetry  of the categories, and we can not obtain another subclass of the conceptually symmetric categories by addition of the new constraints.

 So, finally, we obtained the following ISA (inclusive) hierarchy between different classes of conceptually symmetric categories between their primitive concepts (objects and arrows):\\\\
 \begin{tabular}{|c|c|c|}
   \hline
   Symmetry & Subclass & Added operations/constraints \\
   \hline
    & &\\
   No & Categories & no constraints \\
   &&\\
   \hline
   & & Operation $B_\top$ from arrows into objects, with\\
   Yes & Perfectly symmetric & $B_\top (g\circ f) = B_\top(g)*B_\top(f)$, $~~~B_\top(id_a) \simeq a$\\
   &&\\
   \hline
   & &\\
   Yes & Conceptually closed (CoCC) & Closure $T_e:\textbf{C}\downarrow\textbf{C}\rightarrow\textbf{C}$, with $~T_e^0= B_\top\psi$ \\
   &&\\
   \hline
   & &\\
   Yes & Symmetry-extended & $\tau^{-1}\bullet\tau = \psi$ \\
   &&\\
   \hline
   & &\\
   Yes & Imploded  & $\blacktriangle\circ T_e  \simeq I_{\textbf{C}\downarrow\textbf{C}}$\\
   &&\\
   \hline
 \end{tabular}
 \\
 \\$\square$ \\
\section{Internal Symmetry Group of Transformations}
In previous sections we introduced a kind of \emph{internal symmetry} based on the invariant transformation of the arrows into the "conceptualized" objects. This invariance means that in the categories able for such transformations, this transformation does modify them. Each object obtained as a transformation of some arrow of the category is "equivalent" to this arrow (in an analog way in which each object of a category can be represented by its identity arrow), and the composition of the arrows in this category can be equivalently represented by analog composition of their conceptualized objects. Thus, we consider the categories for which there is this symmetry between primitive categorial concepts: arrows and objects.

Symmetries play a fundamental role in physics because they are
related to conservation laws. This is stated in Noether's theorem
which says that invariance of the action under a symmetry
transformation implies the existence of a conserved quantity.

Albert Einstein once said, about "\emph{the world of our
sense experiences}", and "\emph{the fact that it is comprehensible is a miracle}" (1936, p. 351). A few decades later, another physicist,

Eugene Wigner, wondered about the unreasonable effectiveness of mathematics in the
natural sciences, concluding that "\emph{the miracle of the appropriateness of the language of mathematics for the formulation of the laws of physics is a wonderful gift which we neither understand nor deserve}" (1960, p. 14).
\\ At least three
factors are involved in Einstein’s and Wigner’s miracles: the
physical world, mathematics, and human cognition.
\begin{example} \textbf{Internal symmetries}: \\
In Physics we have also a kind of symmetry, different from the translations/rotations in time-space, as for example in Quantum Physics for a particle represented mathematically by a  complex scalar de Broglie filed $\Psi =
\Phi(t,\overrightarrow{\textbf{r}})\mathrm{e}^{-i\varphi_T}$ representing the wave-packet of a free spin-zero neutral particle, where
$\Phi(t,\overrightarrow{\textbf{r}})$ is a real module of this field and $\varphi_T$ is the De Broglie particle's phase. \index{Internal symmetries}

For the Lie group $G$ of continuous transformations of the
field $\Psi$, we  use the most general one: \index{group
unitary $U(1)$} circle group (unitary group) of transformations, for a given real values $\alpha$,
$U(1) = \{\mathrm{e}^{-i\alpha} | \alpha \in \mathbb{R}\}$, with
unit $1$ (for $\alpha = 0$), multiplication as the group operation,
and for each element $\mathrm{e}^{-i\alpha} \in U(1)$ its inverse
element $\overline{\mathrm{e}^{-i\alpha}} = \mathrm{e}^{i\alpha}$.
From the  standard Quantum Mechanics, two kets $|\Psi\rangle$ and
$|\mathrm{e}^{-i\alpha}\Psi\rangle$ represent the same particle (it is changed only De Broglie particle's phase $\varphi_T$ by the constant value $\alpha$), so
that the particle remains invariant
  by this general group $U(1)$ of transformations.

This is an \emph{internal} transformation because
it transforms the field $\Psi$ into each other
in a way without making reference on time and space. In effect, in \cite{Majk17q},  where $\Phi(t,\overrightarrow{\textbf{r}})$ is considered as the square root of the energy-density of a given massive particle (fermion) localized for each time-instance $t$ in a small 3D volume $V$, for a given Lagrangian density $\L_{free}$ of a free particle  propagating by constant speed $\overrightarrow{\textbf{v}}$ in the vacuum,  that the Noether charge density $J_0$ is equal to the energy-density $\Phi^2$ which remains invariant during such a constant-speed propagation in time-space while the Noether current $\overrightarrow{J}$ is just equal to the constant energy-density flux $\Phi^2\overrightarrow{\textbf{v}}$.
\\$\square$
\end{example}
By considering an analogy with the internal symmetry group in physics $U(1)$, in this section we will provide a formal definition of internal category-symmetry group $ICS(\mathbb{Z})$ of categorial transformations for a more general case of internal symmetry for a perfectly symmetric category $\textbf{C}$ (with a given duality operator $B_T:Mor_\textbf{C} \rightarrow Ob_\textbf{C}$ in (\ref{dualop})  and the partial operation $*$ of its conceptualized (from arrows) objects).

First of all,
based on the isomorphism (\ref{eq:representatib}) between each object $a$ in a perfectly symmetric category $\textbf{C}$ and the conceptualized object obtained from its identity arrow, $\widetilde{id_a} = B_T(id_a)$, for each commutative diagram $f =h\circ g:a \rightarrow b$, we obtain from it another commutative diagram by using these isomorphisms,
\begin{equation} \label{eq:representatiV}
\begin{diagram}
&& \widetilde{id_c}& \lTo^{is_c^{-1}} & c  & \lTo^{is_c} & \widetilde{id_c}&&\\
&\ruTo^{g_1}&&\ruTo_{g}&&\rdTo_{h}&&\rdTo^{h_1}&\\
  \widetilde{id_a} &   \rTo^{is_a} & a &&  \rTo^{f}  && b  &\rTo^{is_b^{-1}}& \widetilde{id_b}\\
\end{diagram}
\end{equation}
so that the external commutative diagram (from the fact that the top composition of the arrows $is_c^{-1}\circ is_c$ is the identity arrow of the conceptualized object $\widetilde{id_c}$, and from the fact that $g_1 = is_c^{-1} \circ g\circ is_a$, $h_1 = is_b^{-1} \circ h\circ is_c$, and defining new arrow
\begin{equation} \label{eq:transfF}
f_1 =  is_b^{-1} \circ f \circ is_a
\end{equation}
 (in the bottom line of the diagram above), we obtain a following \emph{functorial transformation} of the commutative diagram $f = h\circ g:a\rightarrow b$ into derived commutative diagram $f_1 = h_1\circ g_1:\widetilde{id_a} \rightarrow \widetilde{id_b}$ between the conceptualized objects, as we easily show,

$h_1\circ g_1 = (is_b^{-1} \circ h\circ is_c)\circ (is_c^{-1} \circ g\circ is_a) =$

$=  is_b^{-1} \circ h \circ(is_c\circ is_c^{-1}) \circ g\circ is_a $

$=  is_b^{-1} \circ h \circ id_{\widetilde{id_c}} \circ g\circ is_a $

$=  is_b^{-1} \circ h \circ g\circ is_a $

$=  is_b^{-1} \circ f \circ is_a $

$= f_1$
\\
In what follows, we will denote by $\overrightarrow{\textbf{C}}$ the category derived from the category $\textbf{C}$ where each object of $\textbf{C}$, $a \in Ob_C$, is replaced by its identity arrow $id_a:a \rightarrow a$. From the fact that the arrows in $\overrightarrow{\textbf{C}}$ are equal to the arrows in $\textbf{C}$, their composition is equal as that in $\textbf{C}$, that is, given by the same partial operation $\circ$.

From the fact that in the category $\textbf{C}$ we can't have the composition of objects but only of the arrows given by the partial operation $\circ:Mor_\textbf{C}\times Mor_\textbf{C} \rightarrow Mor_\textbf{C}$, we need its derived category $\widetilde{\textbf{C}}$ where the composition of its arrow is given by partial operation $$*:Mor_{\widetilde{\textbf{C}}}\times Mor_{\widetilde{\textbf{C}}} \rightarrow Mor_{\widetilde{\textbf{C}}}$$
such that for each composition of two arrows $g\circ f$ in $\textbf{C}$, we have the composition of their conceptualized objects $\widetilde{f} =B_T(f)$ and $\widetilde{g} =B_T(g)$ written by $\widetilde{g}* \widetilde{f}$. Note that the objects of this derived category $\widetilde{\textbf{C}}$ are the conceptualized objects obtained from the identity arrows of $\textbf{C}$ while the arrows of $\widetilde{\textbf{C}}$ are the conceptualized objects obtained from all arrows in $\textbf{C}$. Moreover, $\widetilde{g}* \widetilde{f}$ is well defined composition of arrows if $dom(\widetilde{g}*) = cod(\widetilde{f}*)$ in $\widetilde{\textbf{C}}$, in according  to composition of the arrows for all categories.\\
So we provide this definition of this derived category \cite{Majk23} by:
\begin{lemma} \label{prop:ConceptTransf}
Let $\textbf{C}$ be  perfectly symmetric category (PSC) with duality operator $B_\top$ satisfying the properties (\ref{eq:SymHomomorph}) and (\ref{eq:representatib}).
Then we define the "conceptual lifting" operator which transforms $\textbf{C}$ into  category of conceptualized objects $\widetilde{\textbf{C}}$,  such that:
\begin{enumerate}
  \item Each arrow in $\widetilde{\textbf{C}}$ is a conceptualized object $\widetilde{f} = B_T(f)$ for $f\in Mor_\textbf{C}$. We denote by $*:Mor_{\widetilde{\textbf{C}}}\times Mor_{\widetilde{\textbf{C}}} \rightarrow Mor_{\widetilde{\textbf{C}}}$ the partial operation for composition of the arrows in $\widetilde{\textbf{C}}$.
  \item Each object in $\widetilde{\textbf{C}}$ is a conceptualized object $\widetilde{id_a} = B_T(id_a)$ for identity arrow $id_a\in Mor_\textbf{C}$, and hence $Ob_{\widetilde{\textbf{C}}}\subset Mor_{\widetilde{\textbf{C}}}$.
\end{enumerate}
Thus, there is a transformation functor $T=(T^0,T^1):\textbf{C}\rightarrow\widetilde{\textbf{C}}$ with the components
\begin{equation}\label{eq:Conc-tran}
T^0 \triangleq  B_\top\tau_I,   ~~~~~~~~~~~~~~T^1 \triangleq B_\top
\end{equation}
where $\tau_I:Id_\textbf{C}\rTo^\centerdot Id_\textbf{C}$ is the identity natural transformation for the category $\textbf{C}$.
\end{lemma}
\textbf{Proof:} It is only necessary to show that $T$ satisfies the properties of a functor for identity arrows and for the composition of arrows:
1. For each object $a$ in $\textbf{C}$, $T^0(a) = \widetilde{id_a}$ which is an identity arrow in $\widetilde{\textbf{C}}$.\\
2. For any well defined composition of arrows $g\circ f$ in $\textbf{C}$ with $T^1(g) = B_T(g)$ and $T^1(f) = B_T(f)$, we have that

$T^1(g\circ f) = B_T(g\circ f)$

$= B_T(g)* B_T(f)$   because $\textbf{C}$ is a PSC, thus holds (\ref{eq:SymHomomorph})

$= T^1(g)*T^1(f)$
\\$\square$\\
In fact, it is well known that each category $\textbf{C}$ can be equivalently replaced by the category $\overrightarrow{\textbf{C}}$ composed by the arrows only in which the objects are represented by their identity arrows, and hence with the functor $M =(\tau_I,id):\textbf{C}\rightarrow \overrightarrow{\textbf{C}}$, where the component of this functor for the arrows is the identity function.
So, if $\textbf{C}$  is perfectly symmetric as well, then we can define also the category $\widetilde{\textbf{C}}$ composed by only conceptualized objects $\widetilde{f} = B_\top(f)$ for each arrow $f$ in $\textbf{C}$ as specified in this proposition, and with the duality functor $D = (B_\top,B_\top):\overrightarrow{\textbf{C}}\rightarrow \widetilde{\textbf{C}}$.

 Thus, by composition we obtain the functor $T = D\circ M:\textbf{C}\rightarrow \widetilde{\textbf{C}}$, so that we have the following correspondence between the conceptually symmetric category $\textbf{C}$ and its two derived dual categories (one for the arrows in $\textbf{C}$ and another for the conceptualized objects obtained from the arrows in $\textbf{C}$):
\begin{equation} \label{fig:dualsym}
\begin{diagram}
 id_a&\rTo^{f}& id_b   && a &\rTo^{f}& b  &&\widetilde{id}_a &\rTo^{\widetilde{f}}          & \widetilde{id}_b    && \\
 & \rdTo_{h = g\circ f}&     \dTo_{g}&&& \rdTo_{h = g\circ f}&     \dTo_{g}&&& \rdTo_{\widetilde{h}= \widetilde{g}*\widetilde{f}}& \dTo_{\widetilde{g}}&&    \\
 &&   id_c   &&&     &   c  &&&   &   \widetilde{id}_c   &&\\
 &\overrightarrow{\textbf{C}}&      &\lTo^M&&    \textbf{C} &    &\rTo^T&&   \widetilde{\textbf{C}}&   &&
\end{diagram}
\end{equation}
The conceptual symmetry of $\widetilde{\textbf{C}}$ is represented by the duality functor $D = (B_\top,B_\top):\overrightarrow{\textbf{C}}\rightarrow \widetilde{\textbf{C}}$, and let us demonstrate that it is just an isomorphism:
\begin{coro} Duality functor $D = (B_\top,B_\top):\overrightarrow{\textbf{C}}\rightarrow \widetilde{\textbf{C}}$ is an isomorphism between the arrow representation category $\overrightarrow{\textbf{C}}$ and the object representation  category $\widetilde{\textbf{C}}$ of a given perfectly symmetric category $\textbf{C}$.

Thus, internal categorial symmetry is represented by this categorial isomorphism $\overrightarrow{\textbf{C}} \simeq \widetilde{\textbf{C}}$ in $\textbf{Cat}$.
\end{coro}
\textbf{Proof}: Let us define the opposite functor $\overline{D} =(\overline{D}^0,\overline{D}^1):\widetilde{\textbf{C}}\rightarrow \overrightarrow{\textbf{C}}$ such that:

1.  for each object $\widetilde{id_a}\in Ob_{\widetilde{\textbf{C}}}$, $\overline{D}^0(\widetilde{id_a}) = id_a$, identity arrow of object $a$ in $\textbf{C}$.

2. for each arrow $\widetilde{f}\in Mor_{\widetilde{\textbf{C}}}$, $\overline{D}^1(\widetilde{f}) = f$, a non identity identity arrow in $\textbf{C}$.
\\\\Thus in $\textbf{Cat}$,

$\overline{D}\circ D = (\overline{D}^0B_T,\overline{D}^1B_T) = Id_{\overrightarrow{\textbf{C}}}$ is the identity endofunctor of $\overrightarrow{\textbf{C}}$, such that for any arrow $f$ in $\overrightarrow{\textbf{C}}$,
\begin{diagram}
  id_a  &&&&\widetilde{id}_a &&&& id_a     \\
  \dTo_{f}&&&&    \dTo_{\widetilde{f}}&&&& \dTo_{f}    \\
   id_b   &&&& \widetilde{id}_b   &&&& id_b\\
  \overrightarrow{\textbf{C}}&&\rTo^D&&    \widetilde{\textbf{C}} &    &\rTo^{\overline{D}}&&   \overrightarrow{\textbf{C}}
\end{diagram}
 and, analogously,

$D\circ\overline{D} = (B_T\overline{D}^0,B_T\overline{D}^1) = Id_{\widetilde{\textbf{C}}}$ is the identity endofunctor of $\widetilde{\textbf{C}}$,
\\
so that we obtain the following commutative diagram in $\textbf{Cat}$:
\begin{equation} \label{fig:universal-arrow40CEqual2}
\begin{diagram}
 \textbf{C}         &&  \\
 \dTo_{M}    &\rdTo^{T} &     \\
  \overrightarrow{\textbf{C}} &\pile{\rTo^{D}\\ \\\lTo_{\overline{D}}}& \widetilde{\textbf{C}}
\end{diagram}
\end{equation}
$\square$\\
In general all constructions of algebraic topology are functorial. In our case of the internal categorial symmetry of a given category $\textbf{C}$, the \emph{internal} transformation of this category have to be its endofunctor.  So, differently from the transformation  as functor $T:\textbf{C}\rightarrow \widetilde{\textbf{C}}$, used in  \cite{Majk23}, which can not generate the symmetry group, from the fact that we can not have the group composition $T \circ T$ because  domain and codomain of $T$ are different categories, we need, by using the functors $M$ and $T$ between the category $\textbf{C}$ and two derived from it categories (of only arrows and of only conceptualized objects), to define the new functors from them into $\textbf{C}$, in order that by composition of them we can obtain an endofunctor of $\textbf{C}$.

However, for such necessity, we can not use the functor $M:\textbf{C}\rightarrow \overrightarrow{\textbf{C}}$, because the arrows of $\overrightarrow{\textbf{C}}$ are equal to the arrows of $\textbf{C}$, and in order to satisfy the functorial property on composition of the arrows, each arrow in $\overrightarrow{\textbf{C}}$ must be mapped into \emph{the same} arrow in $\textbf{C}$. So, generally, the endofunctor has to be the identity endofunctor, which obviously is not any kind of transformation.\\
Consequently, we introduce a new functor for the category $\widetilde{\textbf{C}}$:
 \begin{definition} \label{def:intFunct}
 Let $\textbf{C}$ be a perfectly symmetric category with a given duality operator $B_T$. Then, by using transformation (\ref{eq:transfF}) of diagrams in (\ref{eq:representatiV}), we define the the functor $H = (H^0,H^1):\widetilde{\textbf{C}}\rightarrow\textbf{C}$ by:
 \begin{enumerate}
   \item for each object $\widetilde{id_a}$ of $\widetilde{\textbf{C}}$, $H^0(\widetilde{id_a}) =_{def} \widetilde{id_a}$, that is, $H^0:Ob_{\widetilde{\textbf{C}}}\rightarrow Ob_\textbf{C}$ is an injective function.
   \item for each arrow $\widetilde{f}$ ( conceptualized from the arrow $f:a\rightarrow b$ is not an identity arrow in $\textbf{C}$) of $\widetilde{\textbf{C}}$, $H^1(\widetilde{f}) =_{def} is_{cod(f)}^{-1}  \circ f \circ is_{dom(f)}$.
 \end{enumerate}
 Consequently, we define the endofunctor $E = H\circ T:\textbf{C}\rightarrow\textbf{C}$, so that for each object $a$, from (\ref{eq:Conc-tran}), $E^0(a) = H^0B_\top\tau_I(a) = \widetilde{id_a}$, and each arrow $f:a\rightarrow b$, $E^1(f) = H^1B_\top(f) = is_{b}^{-1}  \circ f \circ is_{a}$.
 \end{definition}
 Let us show that this endofunctor $E$ have and adjunction with the identity endofunctor $Id$ of a perfectly symmetric category $\textbf{C}$:
 \begin{propo} \label{prop:ISA} \textsc{Internal Symmetry Adjunction}: For each perfectly symmetric category $\textbf{C}$ there is the adjunction $(I_\textbf{C},E,\eta,\varepsilon)$ between its identity endofunctor $I_\textbf{C}:\textbf{C} \rightarrow\textbf{C}$ and the endofunctor $E$ of Definition \ref{def:intFunct}, with natural transformation $\eta:I_\textbf{C} \rTo^\bullet E$ such that for each object $c$ its component $\eta(c) = is_c^{-1}:c \rightarrow \widetilde{id_c}$ is an universal arrow from $c$ to $E$, and with natural transformation $\varepsilon:E \rTo^\bullet I_\textbf{C}$ such that for each object $b$ its component $\varepsilon(d) =is_d:\widetilde{id_d}\rightarrow d$ is a couniversal arrow from $I_\textbf{C}$ to $d$.

 Thus, the natural transformation of this adjunction are two natural isomorphisms, $\eta:I_\textbf{C} \simeq E$ and its inverse $\varepsilon:E\simeq I_\textbf{C}$.
 \end{propo}
 \textbf{Proof:}
 For each arrow $f:a\rightarrow b$, we obtain the following two commutative diagrams generated by natural isomorphisms $\eta:I_\textbf{C} \simeq E$ and its inverse $\varepsilon:E\simeq I_\textbf{C}$, with $\eta(a) = is_a^{-1}$ and $\varepsilon(a) = is_a$, so that $\varepsilon(a)\circ \eta(a) =is_a\circ is_a^{-1} = id_a$,
  \begin{diagram}
a&&a =I_\textbf{C}(a) &&\rTo^{\eta(a)}&  \widetilde{id_a}= E(a)       &&\rTo^{\varepsilon(a)}          && a =I_\textbf{C}(a) \\
\dTo^{f}&\mapsto& \dTo^{f}_{=I_\textbf{C}(f)}&&&\dTo^{E(f)}_{= is_b^{-1}\circ f \circ is_a}     && &&     \dTo^{f}_{=I_\textbf{C}(f)}  \\
b&&b =I_\textbf{C}(b) &&\rTo^{\eta(b)}&  \widetilde{id_b}= E(b)       &&\rTo^{\varepsilon(b)}          && a =I_\textbf{C}(b) \\
&&I_\textbf{C}&&\rTo^{\eta}&E&&\rTo^{\varepsilon}&& I_\textbf{C}
\end{diagram}

 Let us show that for any object $a$ in $\textbf{C}$ and arrow $k:c\rightarrow E(d)$ there exists a unique arrow $\underline{k}:I_\textbf{C}(c) \rightarrow d$ (that is, $\underline{k}:c\rightarrow d$), such that the hollowing adjunction diagrams commute:
 \begin{equation} \label{fig:univ-arrowLagr5V}
 \begin{diagram}
 c~         &\rTo^{\eta(c)}          &~ \widetilde{id_c}= EI_\textbf{C}(c)   &&& & c=I_\textbf{C}(c)&\rTo^{I_\textbf{C}(k)} &I_\textbf{C}E(d)  \\
  & \rdTo_{k}&     \dTo_{E(\underline{k})}&&&&\dTo^{\underline{k}} &\ldTo_{\varepsilon(d) = is_d} &     \\
 &   &   \widetilde{id}_d = E(d)   &&&&d
\end{diagram}
\end{equation}
So, from the left-hand diagram above,
\begin{equation} \label{fig:univ-arrowLagr5V1}
k= E(\underline{k})\circ\eta(c) =(is_d^{-1}\circ \underline{k}\circ is_c)\circ_c)\circ is_c^{-1}  = is_d^{-1}\circ \underline{k}
\end{equation}
  Thus, we obtain the composition of arrows

$is_d \circ k = is_d \circ (is_d^{-1}\circ \underline{k}) = $

$= (is_d \circ is_d^{-1})\circ \underline{k} = id_{\widetilde{id_d}}\circ \underline{k}$

$= \underline{k}$, \\
which is just the right-hand commutative diagram above, so defining uniquely the arrow $\underline{k}:c\rightarrow d$ for a given arrow $k$.
\\$\square$\\
\textbf{Remark}: Note that as the consequence of this proposition we obtain that all isomorphic arrows in the commutative diagram (\ref{eq:representatiV}) ate \emph{universal arrows} in $\textbf{C}$, that is, the fundamental properties of the PSC categories are that in it we have the internal symmetry adjunction and these universal arrows that are preserved under comma-transformations symmetry group over all n-dimensional levels. Thus, if a given category $\textbf{C}$ is a PSC, then all its n-dimensional levels $\textbf{C}_n$, for $n\geq 1$, are PSC as well (because the (co)universal arrows and adjunctions are preserved  under the global symmetry).
\\$\square$\\
 Thus, by consecutive repetition of the application of this endofunctor $E$ on the perfectly symmetric category $\textbf{C}$, we obtain this infinite chain of transformations of basic commutative diagram $f=g\circ h:a\rightarrow b$ into commutative diagrams $f_n = g_n\circ h_n$, for $n =1,2,3,...$, where $f_n = (E^1)^n(f)$ (and analogously for $g:a\rightarrow c$ and $h:c\rightarrow b$),
 \begin{equation} \label{fig:dualsymCh}
\begin{diagram}
 a&&    && \widetilde{id_a} &&    &&\widetilde{id}_{\widetilde{id_a}} & && \\
 \dTo_{g}& \rdTo^{f = h\circ g}&     && \dTo_{g_1}& \rdTo^{f_1 = h_1\circ g_1}&     &&\dTo_{g_2}& \rdTo^{f_2= h_2\circ g_2}& &...&    \\
 c&\rTo^{h}&   b   &&\widetilde{id_c}& \rTo^{h_1}    &   \widetilde{id_b}  &&\widetilde{id}_{\widetilde{id_c}} &  \rTo^{h_2} &   \widetilde{id}_{\widetilde{id_b}}  &&\\
 &\textbf{C} &      &\rTo^E&&    \textbf{C} &    &\rTo^E&&   \textbf{C} &  ... &&
\end{diagram}
\end{equation}
Note that in this recursive application of functor $E$, for each $n = 1,2,3,...$, and for example the arrow $g$,
$$g_{n+1} =E^1(g_n) =is_{cod(g_n)}^{-1}\circ g_n \circ is_{dom(g_n)}$$
is an arrow between conceptualized objects of the (n+1)-level, so by application of functorial transformation $E$ we pass into next higher level of conceptualized objects.

This recursive application of the functorial transformation $E$ on this category $\textbf{C}$ can be shown as passing from a given level to the next higher level of commutative diagrams in this infinite chain of commutative diagrams in $\textbf{C}$:
\begin{diagram}
& &a&\rTo^{is_a^{-1}}&&\widetilde{id_a} & \rTo^{is_{\widetilde{id_a}}^{-1}}&&    \widetilde{id}_{\widetilde{id_a}} & \rTo&... \\
~~~~in ~\textbf{C}&& \dTo_{g}&1st &&      \dTo^{g_1=}_{E^1(g)}& & 2nd    &\dTo_{g_2=E^1E^1(g)}& & &...    \\
 &&c&\rTo^{is_c^{-1}}&&   \widetilde{id_c}& \rTo^{is_{\widetilde{id_c}}^{-1}}&    &   \widetilde{id}_{\widetilde{id_c}} &  \rTo &   ...  &\\
~~~~~~in~\textbf{C}\downarrow\textbf{C} &&J(g)&\rTo^{(is_a^{-1}; is_c^{-1})}&&J(g_1) &\rTo^{(is_{\widetilde{id_a}}^{-1}; is_{\widetilde{id_c}}^{-1})}&&J(g_2)&\rTo&... \\
\end{diagram}
so with the isomorphisms of the objects $J(g)\simeq J(E^1(g)) \simeq J(E^1E^1(g)) \simeq ...$ in $\textbf{C}\downarrow\textbf{C}$ as well\footnote{Note that if $\textbf{C}$ is also a CoCC then in it we have also the isomorphisms of conceptualized objets $\widetilde{g} \simeq \widetilde{E^1(g)} \simeq \widetilde{E^1E^1(g)})\simeq ...$ for each arrow $g$ in $\textbf{C}$.}.

Consequently, from Definition \ref{def:newIsomorph} of isomorphism between arrows and the bottom line of the arrows in $\textbf{C}\downarrow\textbf{C}$ above the isomorphisms $g \simeq g_1 \simeq g_2 \simeq \simeq ...$, that is, $g \simeq E^1(g) \simeq E^1E^1(g)\simeq ...$, we obtain that the endofunctor $E$ transforms each commutative diagram into isomorphic to diagram in $\textbf{C}$, that is, for each arrow $g:a \rightarrow c$,

$E^0(a) \simeq a$

$E^1(g) \simeq g$

$E^0(c) \simeq c$
\\
we obtained the isomorphic to it arrow $E^1(g):E^0(a)\rightarrow E^0(c)$ in $\textbf{C}$, and because of this property the endofunctor $E$ is \emph{an invariant (up to isomorphism) transformation}.

These transformations are analog to that of comma-propagation transformations between n-dimensional levels (categories $\textbf{C}_n$ for $n=1,2,3,...$) which generate a \emph{global} symmetry, valid \emph{for every} category $\textbf{C}$. However, in this case all transformations are inside the same perfectly symmetric category $\textbf{C}$, so because of that we call it "internal symmetry" and \emph{local} as well because is not valid for every category but only locally for a category that is perfectly symmetric.\\
Thus, now we are able to define formally this Internal Category-Symmetry Group (ICS):
\begin{definition}\label{def:CatSymGroupINT}
 We define the following infinite abelian ICS (internal category-symmetry group) $ICS(\mathbb{N}) =(\circ, g_0,g_1, g_2, ...)$ of functorial transformations of a given PSC category $\textbf{C}$, with composition operation $\circ$ corresponding to composition of functors, identity element $g_0$ equal to the identity functor $I_\textbf{C}$, i.e., $g_0 = (I_\textbf{C}^0, I_\textbf{C}^1)$,  transformation element $g_1$  is the functor $E =(E^0,E^1):\textbf{C} \rightarrow \textbf{C}$ in Definition \ref{def:intFunct}, and for each natural number $k \in \mathbb{N}$, for $k\geq 2$,
 \begin{equation}\label{eq:CatSymGroup}
 g_k =_{def} \overbrace{g_1\circ ...\circ g_1}^k   ~= (\overbrace{E^0...E^0}^k, \overbrace{E^1...E^1}^k) ~~
 \end{equation}
 such that for any two elements $g_n$ and $g_m$, $n,m \in \mathbb{N}$, we have that $g_k\circ g_m = g_{n+m}$.
\end{definition}
So, if we denote by $(\mathbb{N},+) =(+,0,1,2,...)$ the group of natural numbers with addition  as group operation and $0$ as identity element, then we have simple isomorphism of groups:
\begin{equation}\label{eq:CatSymGroupIso}
 \sigma:(\mathbb{N},+)\simeq ICS(\mathbb{N})
 \end{equation}
 Let us consider now the actions of the internal-symmetry group on the  categorial concepts (objects and arrows of $\textbf{C}$):
 \begin{definition}\label{def:CatSymAction}
 We define the action of the $ICS(\mathbb{N})$ group on any PSC category $\textbf{C}$ written, for any $n \in \mathbb{N}$, by
   \begin{equation} \label{innerProduct}
 (g_n,\overline{c}) \mapsto
 \left\{
    \begin{array}{ll}
    \overbrace{E^0...E^0}^n(\overline{c}), & \hbox{if $~~\overline{c}~~$ is an object of $~~\textbf{C}$}\\
     \overbrace{E^1...E^1}^n(\overline{c}), & \hbox{if $~~\overline{c}~~$ is an arrow of $~~\textbf{C}$}
       \end{array}
  \right.
  \end{equation}
  \end{definition}
Also this passive functorial transformation $E$  of "conceptual lifting" from $\textbf{C}$ into $\widetilde{\textbf{C}}$ is an invariance w.r.t the base category $\textbf{C}$, that is, represents the perfect categorial symmetry.

Note that differently from the comma-propagation passive "Galilean boost" transformations which there exist always, the passive conceptual-lifting transformation there exist only for the perfectly symmetric categories satisfying Definition \ref{def:PerfectSym}.
\section{Conclusions}

We introduced a general phenomena of \emph{internal-symmetry transformations}, in confrontation with  the global categorial symmetry.
Like for the global categorial symmetry also local internal categorial symmetry moves beyond standard geometric or Hamiltonian symmetries (like those of a triangle or quantum Hamiltonians) to a universal framework for categorial symmetry. It provides tools to understand symmetries that operate across different levels of mathematical abstraction, relevant for physics where symmetries (like Lorentz transformations) relate to conservation laws.

In essence, global and local (internal) symmetry groups  are a highly abstract concept exploring how symmetries propagate and interrelate through the layered structure of category theory, offering a deep, universal perspective on what "symmetry" means across mathematics and physics.

About the relationships between the local internal symmetry with the symmetry group $ICS$ of a category $\textbf{C}$ and the global comma-propagation symmetry with the group $CS(Z)$  between its n-dimensional levels \cite{Majk25a}, we conclude the following:
\begin{itemize}
  \item Like the fact that not every category is a CCC (Cartesian Closed Category), not every category is a PSC (Perfectly Symmetric Category). Both of such categories must have a particular set of (co)universal arrows and other properties like adjunctions: each PSC must have the internal symmetry adjunction  defined in Proposition \ref{prop:ISA} and (co)universal arrows that are the components of the natural isomorphisms $\eta$ and $\varepsilon$. \\ Note that such isomorphic (co)universal arrows are direct consequences of the duality representation constraint: for any given object $a$ in this category,  the conceptualized object of its identity arrow is a dual representation of this object $a$, that is, must be isomorphic to it, $is_a:B_\top(id_a) \simeq a$, as defined by (\ref{eq:representatib}) in Definition \ref{def:PerfectSym}.
  \item Like in Physics where the internal symmetry given by the symmetry group of transformations $U(1)$  of some material object represented by de Broglie field $\Psi$ in Example 2, and differently form the time-space translations and rotations, are not geometric transformations by a kind of \emph{abstract} transformations, also  internal categorial transformations are not geometric but abstract functorial transformations local to a given PSC category.

       Thus, such categorial transformations applied to any object $a$ of a category does not change this object of transformation "up to isomorphism" ($a \simeq E(a) \simeq E\circ E(a) \simeq E\circ E\circ E (a)\simeq...$).
  \item Differently from the internal symmetry, the global symmetry is the property \emph{of all} categories. Comma-propagation is a general transformation between n-dimensional levels, applied to the functors and natural transformations as well.
       The comma-propagation symmetry group  $CS(Z)$ of transformations of "Galilean boost" of uniform motion (as in time-space in physics) between  n-dimensional levels  $\textbf{C}_n$  (seen as different "\emph{Galilean reference frames}") of any given category  $\textbf{C}$, defines invariant (conserved) categorial properties (in analogy with conservation of momentum and energy under Galilean boosts in physics).

      If we consider the universal properties of a given base category $\textbf{C}$ (its (co)universal arrows and adjunctions) as a kind of different "Lagrangian" in the Category Theory, then we can consider their invariance under this comma-propagation transformations: a \emph{global symmetry} in this case means that these categorial universal structures are preserved under these general transformations \emph{in all} n-dimensional levels (the law of conservation).

      Thus, global symmetry holds also for each PSC category, so that all n-dimensianal levels of such perfectly symmetric category  are PSC categories (as for a given CCC category $\textbf{C}$, all its n-dimensional levels $\textbf{C}_n$, for $n \geq 1$, are CCC as well) preserving its  (co)universal arrows and internal symmetry adjunction, and generally all commutative diagrams, as we shown in Section 2 for all the cases in the hierarchy of internally symmetric categories.
\end{itemize}
%

%

%

\end{document}